\documentclass[11pt,twoside]{amsart}

\usepackage{enumerate}
\usepackage[hyperindex=true,plainpages=false,colorlinks=false,pdfpagelabels]{hyperref}

\theoremstyle{plain}
\newtheorem{The}{Theorem}[section]
\newtheorem*{The*}{Theorem}
\newtheorem{Pro}[The]{Proposition}
\newtheorem{Lem}[The]{Lemma}

\newtheorem*{Cor*}{Corollary}

\theoremstyle{definition}
\newtheorem*{Def}{Definition}

\theoremstyle{remark} 
\newtheorem{Rem}[The]{Remark}
\newtheorem{Exa}{Example}
\newtheorem*{Rem*}{Remark}

\numberwithin{equation}{section}

\DeclareMathOperator{\End}{End}
\DeclareMathOperator{\Hom}{Hom}
\DeclareMathOperator{\GL}{GL}

\DeclareMathOperator{\Span}{Span}
\DeclareMathOperator{\Id}{Id}

\DeclareMathOperator{\mH}{\mathcal H}
\DeclareMathOperator{\mV}{\mathcal V}
\DeclareMathOperator{\mJ}{\mathcal J}
\DeclareMathOperator{\ed}{d}

\renewcommand{\Im}{\operatorname{Im}}
\renewcommand{\Re}{\operatorname{Re}}

\newcommand{\R}{\mathbb{R}}
\newcommand{\C}{\mathbb{C}}

\renewcommand{\H}{\mathbb{H}}  

\newcommand{\CP}{\mathbb{CP}}
\newcommand{\HP}{\mathbb{HP}}

\begin{document}

\title[Conformal Fibrations of $S^3$ by Circles]{Conformal Fibrations of $S^3$ by Circles}

\author{Sebastian Heller}

\address{Sebastian Heller\\
  Institut f\"ur Mathematik\\
  Universit{\"a}t T{\"u}bingen\\\
  Germany}

\email{heller@mathematik.uni-tuebingen.de}


\date{\today}

\thanks{Author supported by GRK 870 ''Arithmetic and Geometry'' and SFB/Transregio 71}

\begin{abstract} 
 It is shown that analytic conformal submersions of $S^3$ are given by intersections of (not necessary closed) 
complex surfaces with a quadratic real hyper-surface in $\CP^3.$
 A new description of the space of circles in the 3-sphere in terms of a natural bilinear form on the tangent sphere bundle of $S^3$ is given. As an application it is shown that every conformal fibration of $S^3$ by circles is the Hopf fibration up to conformal transformations.
 \end{abstract}

\maketitle


\section{Introduction}
\label{sec:intro}

We consider conformal submersions $\pi\colon S^3\to S^2.$ They give rise to a CR structure defined by the cross product with the oriented unit length vector field $T$ tangent to the fibers. 
The tangent sphere bundle $SP$ of a 3-manifold $P$ inherits a natural CR structure by the work of LeBrun \cite{LeB}.
It turns out (see Proposition \ref{CRPsi}) that the map $T\colon S^3\to SS^3$ is CR holomorphic.
The tangent sphere bundle $SS^3$ can be identified
with a real quadratic CR hyper-surface $Q$ of $\CP^3,$ see Proposition \ref{ss3=Q}. 
From this observation it follows that
 analytic conformal submersions of $S^3$ are locally given by
the intersection
of a complex surface $A$ with $Q\subset \CP^3.$

In the second part of this paper we consider conformal submersions whose fibers are circles. Using the CR theory of the first part we give another proof of the fact that circles in $S^3$ correspond to null-lines in $\CP^3.$ Equivalently, the space of oriented circles is an open dense subset of the $3-$quadric $\mathcal Q\subset\CP^4.$ The degenerated circles, i.e. points, correspond to the boundary in $\mathcal Q.$ We apply this to give an invariant proof
of the fact that a conformal submersion $S^3\to\CP^1$ gives rise to a holomorphic curve $\gamma\colon\CP^1\to\mathcal Q.$ We show that $\gamma$ has degree $1,$ 
and deduce that any conformal submersion of $S^3$ by circles is the 
Hopf fibration up to conformal transformations.

This paper is part of the authors PhD thesis \cite{H}. The author would like to thank his thesis advisor Ulrich Pinkall
\section{CR manifolds}
We start with a short introduction to CR (Cauchy-Riemann) manifolds. We will use a natural CR structure on the tangent sphere bundle of $3-$manifolds to describe
conformal submersions.

\begin{Def}
A CR structure (of hyper-surface type) on a manifold $M$ of dimension $2n-1$ is a complex sub-bundle $$T^{(1,0)}\subset TM\otimes\C$$ of complex dimension $n-1$ which is 
(formally) integrable and satisfies
$$T^{(1,0)}\cap\overline{T^{(1,0)}}=\{0\}.$$
\end{Def} 
The basic example is given by
hyper-surfaces $M\to E$ of complex manifolds. In that case $T^{(1,0)}:=TM\otimes\C\cap T^{(1,0)}E$ and the conditions are satisfied automatically.
 
As in the case of complex manifolds, there is another way to define a CR manifold. Consider the $2n-2$ dimensional bundle
$$H:=\Re\{T^{(1,0)}M\oplus T^{(0,1)}M\}\subset TM$$
with complex structure 
$\mJ\colon H\to H; \mJ(v+\bar v)=i(v-\bar v).$ It is called the Levi distribution of the CR manifold.
The integrability of
$T^{(1,0)}M$ can be written as
\[[\mJ X,Y]+[X,\mJ Y]\in\Gamma(H)\]
and
\[[X,Y]-[\mJ X,\mJ Y]+\mJ[\mJ X,Y]+\mJ[X,\mJ Y]=0.\]
The Levi form $\mathcal L$ is given by the curvature of the
Levi distribution $H:$
$$\mathcal L \colon H\times H\to TM/ H; (X,Y)\mapsto [X,Y]\mod H.$$
This is a skew-symmetric complex bilinear form, and, in the case of oriented $M,$ its index is an invariant of the CR structure. A CR structure is called strictly pseudo-convex if the index is 
extremal, i.e. $\pm(n-1),$ at any point.

A function $f\colon M\to\C$ is called CR function or CR holomorphic
if
$\ed f(\mJ X)=i\ed f(X)$
for all $X\in H.$ CR holomorphic maps between CR manifolds or between CR manifolds and complex manifolds are defined similarly.

\subsection{The tangent sphere bundle}
We now describe the canonical CR structure of the tangent sphere bundle of an oriented conformal $3-$space $(M,[g]).$ This is due to LeBrun \cite{LeB}.
Note that the CR structure here differs from the natural CR structure on the tangent sphere bundle of a Riemannian manifold of arbitrary dimension.
\begin{Def}
The tangent sphere bundle $p\colon SM\to M$ of an oriented conformal $3-$space $(M,[g])$ is the bundle of oriented tangent lines.
\end{Def}
For a metric $g\in[g]$ in the conformal class of $M,$ the unit sphere bundle of $(M,g)$ is
$S_gM:=\{v\in TM\mid g(v,v)=1\}.$ Using $g$ this bundle can be identified with the tangent sphere bundle.
We denote the elements of the tangent sphere bundle by $[v]\in SM.$ Any two representatives $v$ and $\tilde v$ of $[v]\in SM$ are 
positive multiples of each other.

The Levi distribution of the tangent sphere bundle is defined by
$$H_{[v]}:=\{X\in T_{[v]}SM\mid \ed_{[v]} p(X)\perp v\}.$$
At $[v]\in SM,$ the complex structure $\mJ$ on $H$ is
given as follows:
Let $g\in[g]$ be a metric in the conformal class and let $v\in T_pM$ be a vector of length 
$1$ with respect to $g.$ Consider the subspace $V:=v^\perp\subset T_pM.$ 
The tangent space of $T_pM$ can be canonically identified
with $T_pM.$ By using $g$ one obtains an exact sequence
$$0\to V\to H_{[v]}\subset T_{[v]}SM\to V\subset T_pM\to0.$$ On $V$ there is
the canonical complex structure given by the cross-product 
$\mJ(w)=v\times w$ with respect to the orientation and $g.$ Then, there is a unique complex structure $\mJ$ on $H_{[v]}$ making the above sequence complex linear. Of course, this definition does not depend on the chosen metric $g\in[g].$ The CR structure $\mJ$ of the tangent sphere bundle satisfies the integrability conditions, see \cite{LeB}.

\subsection{The tangent sphere bundle of $S^3$}\label{tss3}
We consider the tangent sphere bundle of the conformally flat $3-$sphere in details.
We show that its tangent sphere bundle is a CR hyper-surface of $\CP^3.$

We need to recall some facts of the twistor projection 
$\pi\colon\CP^3\to S^4,$ first. In general the twistor space $P$ of a four dimensional Riemannian manifold $M$ is the bundle of almost complex structures compatible with the metric. The total space of the twistor
projection inherits a canonical complex structure as follows: The fibers are round $2-$spheres, as for a fixed vector $N$ of length $1$ any almost complex Hermitian structure $\mJ$ is
determined by the unit vector $\mJ(N)$ of length $1$ in 
$N^\perp.$ Although the identification with a $2$-sphere is not canonical but rather depends on
the choice of $N,$ the induced complex structure on the vertical space $\mV$ is well-defined. Using the embedding of the space of complex Hermitian structures $P$ into
$\End(TM)$ and the connection induced by the Levi-Civita
connection on $\End(TM),$ one obtains the horizontal bundle 
as the tangents to parallel curves of complex structures.
 This horizontal space 
projects isomorphically onto the tangent space of $M.$
The complex structure $\mJ$ lifts to the horizontal space 
$\mH_{\mJ}$ at $\mJ.$ Combined with the complex structure on $\mV$ it defines the canonical almost complex structure on the twistor space.

There is a nice description in the case of the round sphere $S^4.$ Consider $\H^2$ 
as a right vector space. Via the complex structure given by right multiplication
with $i$ we can identify $\H^2\cong \C^4.$ The twistor projection is given by
$$\pi\colon\CP^3\to S^4\cong\HP^1; e\in\CP^3\mapsto l=e\H\in\HP^1,$$
which maps complex lines $e$ to quaternionic lines
$e\H=e\oplus ej.$ 

The round  sphere
$S^3\subset\HP^1$ is the space of isotropic lines of the indefinite Hermitian metric $\langle,\rangle$ given by
$\langle v,w\rangle=\bar v_1w_1-\bar v_2w_2.$ Then 
the real quadric $Q:=\pi^{-1}(S^3)\subset\CP^3$ is the set of isotropic lines of 
$$(,)\colon\C^4\times\C^4\to\C; (z,w)=\bar z_1w_1+\bar z_2w_2-\bar z_3w_3-\bar z_4 w_4$$ in $\C^4.$ 
Evidently that is a $CR$ hyper-surface of $\CP^3.$ Let
$N$ be the oriented unit normal vector field of $S^3\to S^4\cong\HP^1.$
There is a one-to-one correspondence between almost complex Hermitian structures along $S^3\subset\HP^1$  and elements of the unit sphere bundle via 
\begin{equation}\label{jt}
\mJ\in\End(T_l\HP^1)\mapsto [\mJ(N)]\in S_lS^3.
\end{equation}
\begin{Pro}\label{ss3=Q}
The CR structure of the tangent sphere bundle of $S^3$ is given by the (restricted) twistor fibration $Q\subset\CP^3\to S^3\subset S^4$
via the map \eqref{jt}.
\end{Pro}
\begin{proof}
First, we need a better understanding of the map in \eqref{jt}. Fix the embedding
$x\in S^3\subset\H\mapsto[x:1]\in\HP^1.$ Using $\langle,\rangle$ we identify
\[T_{[x:1]}\HP^1\cong\Hom_\H([x:1];\H^2/[x:1])=\Hom([x:1],[x:1]^\perp)\]
For the oriented unit normal vector $N_x$ at $x\in S^3\to\HP^1$ we obtain the 
quaternionic homomorphism which maps $(x,1)$ to $\frac{1}{2}(x,-1).$ 
The tangent bundle of $S^3$ is given by $TS^3\cong S^3\times\Im\H$ via left
 translation. The tangent vector $\mu\in\Im\H$ at $x\in S^3$ corresponds to
quaternionic linear mapping which assigns $\frac{1}{2}(x,-1)\mu$ to $(x,1).$ 
Therefore, the almost complex hermitian structure of 
$T_{[x:1]}\HP^1\cong\Hom([x:1],[x:1]^\perp),$ which maps
$N$ to the vector given by $\mu\in\Im\H,$ is the complex structure 
$\mJ\in\End([x:1])$ with
$\mJ(x,1)=(x,1)\mu$ via pre-composition. The complex line $e$ corresponding
to the tangential vector $\mu$ at $x$ is determined to be
\begin{equation}\label{emu}
e=\{(x,1)\lambda\mid \mJ(x,1)\lambda=(x,1)\lambda i\}
=\{(x,1)\lambda\mid \mu\lambda=\lambda i\}.
\end{equation}
It remains to show that the mapping in \eqref{jt} is $CR$ holomorphic. Note that the vertical spaces of
both bundles are contained in the respective Levi distributions. Moreover, the complex structures restricted to them are clearly the same by the description of the complex structure on the vertical
space of the twistor space. The intersection of the Levi
distribution and the horizontal space of $Q$ is given by 
$\mJ(\hat N)^\perp,$
where $\hat N$ is the (horizontal lift of the) oriented unit normal field of $S^3\to\HP^1.$ The almost complex structure on this space is given by
$\mJ(v)=\mJ(N)\times v.$ Therefore it remains to show that the horizontal bundles of $SS^3$ and $Q$ are identical
via \eqref{jt}, where the horizontal bundle
of $SS^3$ is given by the Levi-Civita connection of the round
metric of $S^3.$
This is easily deduced from the definition together with 
\[\nabla_X(\mJ(N))=(\nabla_X\mJ)(N)+\mJ(\nabla_XN)\]
and the fact that the unit normal field $N$ of the totally geodesic $S^3\subset S^4$ is parallel.
\end{proof}

\section{Conformal Submersions}
Conformal submersions are generalizations of Riemannian submersions as conformal immersions to isometric ones:
\begin{Def}
A submersion $\pi\colon P\to M$ between Riemannian manifolds 
$(P,g)$ and $(M,h)$ is called
conformal if for each $p\in P$ the restriction of the differential to the complement of its kernel
$$\ed_p\pi\colon\ker\ed_p\pi^\perp\to T_{\pi(p)}M$$
is conformal.
\end{Def}
Of course, the definition does not depend on the representative of a conformal class. The Hopf fibration is a basic example. We call
$\mV=\ker\ed\pi$ the vertical space and its orthogonal complement 
$\mH=\ker\ed\pi^\perp$ the horizontal space. It is easy to prove the following Proposition:
\begin{Pro}\label{cofo}
A submersion between Riemannian manifolds is conformal if and only if
the differential at a point, restricted to the horizontal space, is conformal and
the Lie derivative of any vertical field on the metric restricted to the horizontal space is conformal, that means
for all $V\in\Gamma(\ker\ed\pi)$ and $X,Y\in\ker\ed\pi^\perp$ 
$$\mathcal L_Vg(X,Y)=v(V)g(X,Y),$$
where $v(V)$ does only depend on $V.$
\end{Pro}
\subsection{The CR structure of a fibered conformal $3-$manifold}
Let $\pi\colon P\to M$ a submersion of a conformal oriented $3-$space to an oriented surface. We define a CR structure on $P$ as follows: We set 
$H:=\mH=\ker\ed\pi^\perp$ and define $\mJ\colon H\to H$ to be the 
rotation by $\frac{\pi}{2}$ in positive direction. We extend $\mJ$ to
$TP$ by
setting $\mJ(T)=0$ for $T\in\ker\ed\pi$ and call $\mJ\in\End(TP)$
the CR structure on $P.$ Note that $\mJ$ can be computed as
follows: Fix a metric in the conformal class and the unit length
vector field $T$ in positive fiber direction, then $\mJ(X)=T\times X.$\\

There is the following characterization of conformal submersions of 3-manifolds due to Pinkall. A similar
approach was used by \cite{B-W}.
\begin{Pro}\label{CRPsi}
A submersion $\pi\colon P\to M$ between an oriented conformal 
$3-$space to an oriented surface
is conformal for a suitable complex structure on $M$ if and only if
the map $\Psi,$ through which each point in P is assigned to the oriented fiber direction
$[T]\in SP,$ is CR holomorphic.
\end{Pro}
\begin{proof}
Fix a Riemannian metric $g\in[g]$ in the given conformal class on $P.$ By definition of the
CR structure of $SP,$ $\ed\Psi$ maps the Levi distribution of $M$
into that of $SP.$ Moreover it is obvious that the composition of the
projection onto the horizontal part of the Levi distribution, which
depends on the choice of $g,$ and $\ed\Psi$ commute with the
$\mJ'$s. But the vertical part of
$\ed\Psi,$ which again is depending on $g\in[g],$ is given by $\nabla T,$ where $T$ is the unique vector field in positive fiber direction of length $1.$ It remains to show that $\pi$ is conformal if and only if 
$$\mJ\nabla_XT=T\times\nabla_XT=\nabla_{T\times X}T
=\nabla_{\mJ X}T$$ for all $X\perp T.$ But this equation is 
clearly equivalent to the characterization of Proposition \ref{cofo}.
\end{proof}
In the case of $P=S^3$ we have seen that $SS^3$ is
a CR hypersurface of $Q\to\CP^3,$ hence $\Psi$ is a CR immersion of 
$S^3$ with the induced CR structure into complex projective 
space. 
\begin{Pro}\label{analytic_submersion}
Let $\pi\colon S^3\to S^2$ be an analytic conformal submersion. Then
the image of the oriented fiber tangent map $\Psi$ is the intersection of a complex surface $A\subset\CP^3$ and $SS^3\cong Q\subset\CP^3.$
\end{Pro}
\begin{proof}
In the case of an analytic submersion the CR structure and the
map $\Psi$ are analytic, too. According to the real analytic embedding theorem,
see \cite{Bo}, $S^3$ can be locally CR embedded into $\C^2.$ Using the analyticity again, the map $\Psi$ can be extended uniquely to a holomorphic map on a neighborhood of 
$S^3\subset\C^2.$
It remains to prove that the images of the extensions do match to a complex surface. This follows easily by the uniqueness of the extensions.	
\end{proof}
Conversely, the intersection of a surface $A\subset\CP^3$ with $Q$ 
defines an oriented $1-$dimensional distribution (locally) which gives
rise to conformal foliations or, if one restricts oneself to small open subsets, to conformal submersions.
\begin{Exa}
The easiest examples of surface in $\CP^3$ are complex planes.
Consider the plane given by \[A:=\{[z]\in\CP^3\mid z_4=0\}.\]
As $A$ intersects each tangent sphere $Q_p\cong\CP^1$ over a point 
$p\in S^3$ exactly once, it defines a distribution on the whole $S^3.$ In fact the integral curves
of this distribution are the fibers of the conjugate Hopf fibration. To prove this note that 
$$A\cap Q=\{[z_1:z_2:1:0]\mid \ \mid z_1\mid^2+\mid z_2\mid^2=1\}.$$
That set corresponds to the complex structure which maps 
$1=N\in{TS^3}^\perp$ to $i\in\Im\H=TS^3,$ compare with \eqref{emu}. 
\end{Exa}

\section{Conformal fibrations by circles}
We prove that every conformal fibration of $S^3$ by circles is the Hopf fibration up to conformal
transformations. For details on conformal geometry of $S^3$ cf. \cite{KP}.

\subsection{The space of circles in $S^3$}
An oriented circle in $S^3\subset\R^4$ is the (nonempty) intersection of $S^3$ with an 
oriented affine $2-$plane $A.$ To describe the space of circles we introduce a $2-$form 
$\Omega$ as follows. Let $Q\subset\CP^3$ be the tangent sphere
bundle of $S^3\subset\HP^1,$ and consider the indefinite metric $\langle,\rangle$ 
on $\H^2$ as in Section \ref{tss3}. We decompose $\langle,\rangle$ into $(1,i)$ and 
$( j, k)$ parts
$$\langle,\rangle=(,)+\Omega j.$$
We can regard $(,)$ as a Hermitian form on $\C^4$ and $\Omega$ as 
a complex $2-$form on $\C^4.$ Note that $\Omega$ is 
non-degenerated and $0\neq\Omega\wedge\Omega\in\Lambda^4\C^4.$
\begin{Lem}
Let $\Omega$ be as above, and $[T],[S]\in Q\subset\CP^3$
lying over different points of $S^3.$
Then there exists an oriented circle such that $[T]$ and $[S]$ are
tangent to it in positive direction if and only if $\Omega(T,S)=0.$
\end{Lem} 
\begin{proof}
We are only going to prove that $\Omega(T,S)=0$ is a necessary condition. The
converse direction follows in an analogous manner.
Every conformal transformation of 
$S^3\subset\HP^1$ is given by a quaternionic linear map $\Phi\in\GL(2,\H)$ with $\Phi^*\langle,\rangle=\langle,\rangle.$
Then $\Phi$ also acts on $Q\subset\CP^3.$
In fact, this action is given by the differential of the conformal
transformation $\Phi$ acting on the tangent sphere bundle. 
Therefore we can assume that $[T],[S]\in Q$ are tangent in
the positive direction to the circle given by 
$\{[z:1]\in\HP^1\mid z\in S^1\subset\Span(1,i)\}$ with positive 
oriented tangents given by the set
$\{[z:0:1:0]\in\CP^3\mid z\in S^1\subset\C\}.$
This set is the intersection of $Q$ with the projectivization of the null-plane 
$\Span(e_0,e_3)$ with respect to $\Omega.$
\end{proof}
\begin{Rem}\label{nulllines}
The proof of this lemma shows that every circle is given by the
intersection of $Q$ with a contact line in $\CP^3,$ i.e. the projectivization of a
$2-$plane on which $\Omega$ vanishes. This intersection
already determines the contact line. Conversely, the intersection of a contact line
with $Q$ is the set of tangents of a circle or, in the case of a plane
spanned by $v$ and $vj,$ the tangent sphere over a point $v\H\in S^3\subset\HP^1.$ The latter case
corresponds to degenerate circles. 
\end{Rem}
The space of contact lines in $\CP^3$ is given by the complex 
$3-$quadric $\mathcal Q^3\subset\CP^4$ via the Klein correspondence: Consider the 
$5-$dimensional subspace $W$ of $\Lambda^2(\C^4)$ defined as
$$W:=\{\omega\in\Lambda^2(\C^4)\mid \Omega(\omega)=0\}.$$
It is easy to verify that $\frac{1}{2}\Omega\wedge\Omega$ is 
non-degenerated on $W.$ Thus the space of null-lines with respect to 
$\frac{1}{2}\Omega\wedge\Omega$ is a $3-$dimensional
quadric $\mathcal Q^3\subset PW.$ A short computation  shows that every null-line in $W$ has the
shape $[v\wedge w],$ hence it corresponds to the contact line $P\Span(v,w)$ in 
$\CP^3.$ Conversely, the contact line $P\Span(v,w)$ in $\CP^3$ determines
the null-line $[v\wedge w]$ as an element of $\mathcal Q^3\subset W.$ Altogether we obtain the well-known:
\begin{The}
The space of oriented circles in $S^3$ is given by $\mathcal Q^3\subset W.$
\end{The}
We need to characterize points lying on a circle in terms of null lines. To do so we introduce the real structure
$$\sigma\colon \Lambda^2(\C^4)\to\Lambda^2(\C^4); v\wedge w=(vj)\wedge(wj)$$
on $\Lambda^2(\C^4).$ This is a real linear map with $\sigma^2=\Id.$ In fact, 
$\sigma$ can be restricted to the space of $\Omega-$null $2-$vectors: $\sigma\colon W\to W.$
Moreover, the real points $[\omega]\in \mathcal Q^3\cap PW,$ i.e. $\sigma(\omega)=\omega,$ 
have exactly the form $\omega=v\wedge vj.$ This means that they can be regarded as quaternionic lines $v\H$ or points in $S^3\subset\HP^1,$ compare with Remark \ref{nulllines}.
\begin{Pro}\label{interperp}
A point $p=v\H\in S^3\subset\HP^1$ represented
by the null-line  $[\omega]=[v\wedge vj]\in \mathcal Q^3\subset PW$ lies on an oriented
circle $k$ represented by $[\eta]\in \mathcal Q^3\subset PW$ if and only if  
$$\frac{1}{2}(\Omega\wedge\Omega)(\omega\wedge\eta)=0.$$
\end{Pro}
\begin{proof}
If $p$ lies on the circle $k$ we choose $[v],[w]\in\CP^3$ tangent to $k$ such that $k$ is given by $[\eta]=[v\wedge w]\in \mathcal Q^3$
and $p$ is given by $[\omega]=[v\wedge vj]\in \mathcal Q^3.$ Then
$$\frac{1}{2}(\Omega\wedge\Omega)(\omega\wedge\eta)=\frac{1}{2}(\Omega\wedge\Omega)(v\wedge vj\wedge v\wedge w)=0.$$

Conversely, let $\omega=v\wedge vj$ and $\eta=w\wedge \tilde w$
representing $p$ and $k$ in $\mathcal Q^3$ with
$\frac{1}{2}(\Omega\wedge\Omega)(\omega\wedge\eta)=0.$ As
$\frac{1}{2}\Omega\wedge\Omega\neq0\in\Lambda^4(\C^4)$ we deduce that $v,vj, w$ and $\tilde w$ are linear dependent. If
$\eta=w\wedge\tilde w$ represents a degenerate circle, we 
can assume $\tilde w=wj,$ and $v,vj, w, wj$ are (complex)
linear dependent if and only if they lie on the same quaternionic 
line. That means $v\H=w\H\in S^3\subset\HP^1$ and $p=k.$ 
Otherwise, the complex plane $\Span(w,\tilde w)$ would intersect 
the complex plane $\Span(v,vj).$ For an element 
$\hat w\in\Span(w,\tilde w)\cap\Span(v,vj),$  
$[\hat w]\in Q\subset\CP^3$ is tangent to the
circle  $k$ at $p.$
\end{proof}

Let
$\pi\colon S^3\to\CP^1$ be a conformal fibration by circles. The
fibers have the following induced orientation: Let $T$ be tangent to
the fiber and $A,B\perp T$ such that $\pi_*A\wedge \pi_*B>0$
represents the orientation of the Riemannian surface $\CP^1.$ We
say that $T$ is in positive direction to the fiber if 
$T\wedge A\wedge B>0.$ 

We give a more invariant proof of a theorem of Baird \cite{B}:
\begin{The}\label{gamma}
Let $\pi\colon S^3\to\CP^1$ be a conformal fibration such that all fibers are circles with the induced orientation. Then the curve 
$\gamma\colon\CP^1\to \mathcal Q^3$
which maps each $p\in\CP^1$ to the oriented circle 
$\pi^{-1}(p)\in \mathcal Q^3\subset PW$ is 
holomorphic. 
\end{The}
\begin{proof}
Fix a metric on $S^3.$ Let $T$ be the unit length tangent 
vector in positive fiber direction. The fibers of $\pi$ are the closed integral curves of
the flow $\Phi$ of $T.$ There exists 
$s>0$ such that for all $\phi:=\Phi_s\colon S^3\to S^3$ is a fixed point
free diffeomorphism. As an example, one could take any $s$ with
$s<length(\pi^{-1}(p))$ for all $p\in S^2.$ Let $A, B$ be the 
horizontal lifts of vector fields defined on the base space such that $(T,A,B)$ is an
oriented orthonormal frame field. Because of $\pi_*[A,T]=0=\pi_*[B,T]$ there exist functions 
$\lambda,\mu$ such that $\phi_*A=A+\lambda T$ and 
$\phi_*B=B+\mu T.$

We have already seen that the tangent map 
$\Psi\colon S^3\to Q\subset\CP^3$ is CR-holomorphic.
Let $\psi$ be a local CR holomorphic lift to $\C^4.$ We define
$\chi=\psi\circ\phi.$ Then $\gamma$ is given by
$$\gamma(\pi(p))=[\psi(p)\wedge\chi(p)]=[\psi(p)\wedge T\cdot\psi (p)]\in \mathcal Q^3\subset PW.$$
We set $\hat\gamma=\psi\wedge \chi\colon U\subset S^3\to W\subset\Lambda^2(\C^4)$ 
and compute
\begin{equation*}\begin{split}
A\cdot\hat\gamma&=A\cdot\psi\wedge\chi+\psi\wedge A\cdot\chi
=A\cdot\psi\wedge\chi+\psi\wedge (A\cdot \psi)\circ\phi+
\psi\wedge (\lambda T\cdot \psi)\circ\phi\\
B\cdot\hat\gamma&=B\cdot\psi\wedge\chi+\psi\wedge B\cdot\chi
=B\cdot\psi\wedge\chi+\psi\wedge (B\cdot \psi)\circ\phi+
\psi\wedge (\mu T\cdot \psi)\circ\phi.\\
\end{split}\end{equation*}
From $\mJ A=B$ and the fact that $\Psi$ is CR holomorphic we
deduce that $$iA\cdot\psi=B\cdot\psi\mod \Psi.$$ As 
$[\hat\gamma]=[\psi\wedge T\cdot \psi]$ we have
$(\mJ A)\cdot (\gamma\circ\pi)=i(A\cdot(\gamma\circ\pi)).$
This means that $\gamma\circ\pi$ is CR holomorphic. 
Since $\ed\pi$ maps the Levi distribution of  $S^3$ CR holomorphically
onto the tangent space of $\CP^1$ and $\gamma\circ\pi$ is CR holomorphic, it is evident that
$\gamma$ must be holomorphic.
\end{proof}

\begin{Pro}\label{grad1}
Let $\pi\colon S^3\to S^2$ be a conformal fibration with circles as fibers.
Then the holomorphic curve $\gamma\colon\CP^1\to \mathcal Q^3$ given
by Theorem \ref{gamma} has degree $1.$
\end{Pro}
\begin{proof}
Let $n$ be the degree of 
$\gamma\colon\CP^1\to \mathcal Q^3\subset PW.$ We take a
basis $(e_0,.., e_4)$ of $W$ consisting of real points, i.e. 
$\sigma(e_i)=e_i,$ such that 
$$(,):=\frac{1}{2}\Omega\wedge\Omega=-e_0^*\otimes e_0^*
+e_1^*\otimes e_1^*+\cdot\cdot\cdot+e_4^*\otimes e_4^*,$$ 
and such that we 
have $\gamma(0)=[e_3+ie_4]$ for a suitable holomorphic 
coordinate $z$ on $\CP^1.$ By changing the holomorphic coordinate $z$ we can assume that the point
$[e_0+e_3]\in S^3\subset \mathcal Q^3$ lies on the circle 
$\gamma(\infty).$
We take a holomorphic lift 
$$\hat\gamma(z)=v_0+zv_1+\cdot\cdot\cdot+z^nv_n$$
of $\gamma$ to $W$ with $v_0=e_3+ie_4.$ 
Since every point $p\in S^3$ lies on exactly one circle, namely $\gamma(\pi(p)),$ it results from Proposition \ref{interperp} that $\gamma$
intersects each hyperplane $Pv^\perp\cap \mathcal Q^3$ at exactly one
point, where $[v]\in S^3\subset\R P^4\subset PW.$ But every 
hyperplane must be intersected $n-$times by $\gamma$ counted with multiplicities, hence we know that $\gamma$ intersects $Pv^\perp$ in exactly one point with order $n.$
Therefore  
$$Pv_0^\perp\cap S^3=\cup_{\varphi\in[0,2\pi]}[e_0+\cos\varphi e_1+\sin\varphi e_2]$$
proves that
$$v_1,v_2,..,v_{n-1}\in\cap_{\varphi\in[0,2\pi]}(e_0+\cos\varphi e_1+\sin\varphi 
e_2)^\perp=\Span\{e_3,e_4\}.$$
Using $(\hat\gamma(z),\hat\gamma(z))=0$ this implies inductively 
$v_1=\mu_1v_0,..,v_{n-1}=\mu_{n-1}v_0$ for suitable $\mu_i\in\C,$
and 
$(v_0,v_n)=(v_n,v_n)=0.$ As $(v_n,e_0+e_3)=0$ we see that
$v_n=ae_0+be_1+ce_2+a(e_3+ie_4)$ with $a^2=b^2+c^2$ and 
$bc\neq0.$
For the holomorphic coordinate $\omega=1/z$ we obtain another
holomorphic lift
$\tilde\gamma(\omega)=(ae_0+be_1+ce_2+a(e_3+ie_4))+(\omega\mu_{n-1}+\cdot\cdot\cdot+\omega^n)v_0$
of $\gamma,$
thus $\gamma$ intersects the hyperplane 
$P(e_0+e_3)^\perp$ in  
$z=\infty$ $n-$times if and only if $\mu_1=\cdot\cdot\cdot=\mu_{n-1}=0.$

First we consider the case $a=0,$ i.e. $v_n=be_1+ce_2$
with $b^2+c^2=0.$ After a change of the coordinate $z$ on 
$\CP^1$ we obtain the form 
$\hat\gamma(z)=e_3+ie_4+z^n(e_1+ce_2)$
where $c^2=-1.$ Hence the circle 
$\gamma(1)$ contains the point
$[e_0+\frac{1}{\sqrt 2}(e_1-e_3)]\in S^3.$ The intersection of $\gamma$ and the hyperplane 
$P(e_0+\frac{1}{\sqrt 2}(e_1-e_3))^\perp$ is, at $z=1,$ only of 
order $1$ because
$$(\hat\gamma'(1),e_0+\frac{1}{\sqrt2}(e_1-e_3))=\frac{n}{\sqrt2}\neq0.$$
Since $\gamma$ must intersect $(e_0+\frac{1}{\sqrt2}(e_1-e_3))^\perp$ at $z=1$ with order
$n$ we see that $n=1$ when
$a=0.$

If $a\neq0$ we can change  the coordinate $z$ by a factor such
that $[e_0-e_3]$ lies on $\gamma(1).$ This implies 
$0=(\hat\gamma(1),e_0-e_3)=-a-(a+1).$
The curve $\gamma$ intersects $P(e_0-e_3)^\perp$ at $z=1$ with multiplicity $n,$  
which implies for $n>1$ that $0=(\hat\gamma'(1),e_0-e_3)=-a-(a+n)$ contradicting
$0=-2a+1.$ Thus $n=1.$
\end{proof}
With this result it is easy to prove the following
theorem:
\begin{The}\label{confcirc}
Up to conformal transformations of $S^2$ and $S^3,$ every conformal fibration
of $S^3$ by circles is the Hopf fibration.
\end{The}
\begin{proof}
We will show that for each curve $\lambda\colon\CP^1\to \mathcal Q^3$
which is given by a conformal submersion via Theorem \ref{gamma} there
is a projective isomorphism $\mathcal Q^3\to \mathcal Q^3$ resulting from a
conformal transformation of $S^3,$ such that $\lambda$ is mapped to the curve $\gamma\colon\CP^1\to \mathcal Q^3$ given 
by $\gamma([z:w])=[z(e_1+ie_2)+w(e_3+ie_4)],$
where $e_0,..,e_4$ is a basis of $W$ as in the proof of
Proposition \ref{grad1}. Using this basis,  
$S^3\subset \mathcal Q^3$ and the sphere in the light-cone model can be identified, and 
we see that the real orthogonal transformations 
$\Phi\colon W\to W,$ i.e. 
$\sigma\circ\Phi=\Phi\circ\sigma,$ and $\Phi^*(,)=(,),$ with
$(\Phi(e_0),e_0)<0$ are exactly the conformal transformations of $S^3.$

As in the proof of Proposition \ref{grad1} we can take a
holomorphic coordinate $z$ on $\CP^1$ such that, after a 
conformal transformation of $S^3$ and the induced transformation of $W$ we have $\lambda(0)=[e_3+ie_4],$ and $[e_0+e_3]$ lies on
$\lambda(\infty).$ Thus $\hat\lambda(z)=e_3+ie_4+zb(e_1\pm ie_2),$ where $b\in\C.$ 
The curves $\lambda$ and $\gamma$ coincide after applying the conformal transformation 
$\tilde z=1/bz$ on $S^2$ and the conformal transformation on $S^3$ given by the linear
isometry of $W$ with
$e_2\mapsto\pm e_2;\ e_i\mapsto e_i,\ i\neq2.$
\end{proof}


\end{document}